\documentclass[11pt,letterpaper]{article}
\usepackage[margin=1.5in]{geometry}
\usepackage{amsmath,times}
\usepackage{amssymb,amsmath}
\usepackage{graphicx}
\usepackage{amsthm}
\usepackage{multicol}
\usepackage[all]{xy}
\usepackage{cancel}
\usepackage{tikz}
\usepackage[utf8]{inputenc}

\title{The relative rank of the endomorphism monoid of a finite $G$-set}
\author{Alonso Castillo-Ramirez\footnote{Email: alonso.castillor@academicos.udg.mx}, \  Ram\'on H. Ruiz-Medina  \\[1em]
\small{Centro Universitario de Ciencias Exactas e Ingenier\'ias}, \\ 
\small{Universidad de Guadalajara, Guadalajara, M\'exico.}}

\newtheorem{theorem}{Theorem}[]
\newtheorem{lemma}[theorem]{Lemma}
\newtheorem{remark}[theorem]{Remark}
\newtheorem{claim}[theorem]{Claim}
\newtheorem{corollary}[theorem]{Corollary}
\newtheorem{proposition}[theorem]{Proposition}
\newtheorem{example}[theorem]{Example}
\newtheorem{definition}[theorem]{Definition}

\newcommand{\CA}{\mathrm{CA}}

\newcommand{\Rank}{\mathrm{Rank}}
\newcommand{\Sym}{\mathrm{Sym}}

\newcommand{\End}{\mathrm{End}}
\newcommand{\Aut}{\mathrm{Aut}}
\newcommand{\Stabs}{\mathrm{Stabs}}
\newcommand{\Tran}{\mathrm{Tran}}
\newcommand{\id}{\mathrm{id}}
\newcommand{\Conj}{\mathrm{Conj}}

\begin{document}

\maketitle

\begin{abstract}
For a group $G$ acting on a set $X$, let $\End_G(X)$ be the monoid of all $G$-equivariant transformations, or $G$-endomorphisms, of $X$, and let $\Aut_G(X)$ be its group of units. After discussing few basic results in a general setting, we focus on the case when $G$ and $X$ are both finite in order to determine the smallest cardinality of a set $W \subseteq \End_G(X)$ such that $W \cup \Aut_G(X)$ generates $\End_G(X)$; this is known in semigroup theory as the \emph{relative rank} of $\End_G(X)$ modulo $\Aut_G(X)$.  \\


\end{abstract}


\section{Introduction}

For any group $G$, a \emph{$G$-set} is simply a set $X$ on which $G$ acts; this is, there exists a function $\cdot : G \times X \to X$ such that $e \cdot x = x$, for all $x \in X$, and $g \cdot (h \cdot x) = gh \cdot x$, for all $x \in X$, $g,h \in G$. In the context of semigroup theory, $G$-sets are also known as \emph{$G$-acts}. A \emph{$G$-equivariant transformation} of $X$, or a \emph{$G$-endomorphism} of $X$, is a function $\tau : X \to X$ such that $\tau(g \cdot x) = g \cdot \tau(x)$, for all $g \in G$, $x \in X$. In general, $G$-equivariant maps are the standard morphisms considered in the category of $G$-sets (\cite[p. 86]{AB14}), and they are widely used in various areas of mathematics such as equivariant topology, topological dynamics, representation theory, and even in statistical inference, via the so-called equivariant estimators. 

In this setting, two basic objects, which have been previously studied from a semigroup-theoretic perspective (e.g., see \cite{Bulman, Bulman2, Fleischer,Knauer}), are the monoid $\End_G(X)$, consisting of all $G$-equivariant transformations of $X$ equipped with composition, and its group of units $\Aut_G(X)$, consisting of all bijective $G$-equivariant transformations of $X$. In particular, \emph{free} $G$-sets (i.e., $G$-sets on which all point stabilizers are trivial) and their endomorphisms have gained special interest, as the former are examples of \emph{independence algebras} (see \cite{Ara2,Dandan, Gould1, Gould2,Gray2}).   

Our work in this paper is related with the concept of \emph{rank} of a monoid $M$, denoted by $\Rank(M)$, which is the smallest cardinality of a generating set of $M$. For any subset $N \subseteq M$, define the \emph{relative rank of $M$ modulo $N$} by 
\[ \Rank(M : N )  := \min\{|W| :  \langle W \cup N \rangle = M  \}. \]
Finding the relative rank of a finite monoid $M$ modulo its group of units $U$ is a natural and important question in semigroup theory (e.g., see \cite{Ara, Ara3, Dimi2, Gray, Higgins,HRH98,R94}), as it is linked to the rank of $M$ via the formula
\[ \Rank(M)= \Rank(U)+ \Rank(M:U).  \]

Our main theorem determines the relative rank of $\End_G(X)$ modulo $\Aut_G(X)$ when $G$ and $X$ are both finite. Before stating it, we introduce some notation: let $[H]$ be the conjugacy class of a subgroup $H \leq G$, let $\Conj_G(X)$ be the set of conjugacy classes of point stabilizers in $X$, let $N_G(H)$ be the normalizer of $H$ in $G$, and let $\mathcal{O}_{[H]}$ be the set of all $G$-orbits whose point stabilizer is conjugate to $H$. 
\begin{theorem}\label{main-intro}
Let $G$ be a finite group acting on a finite set $X$. Let $[H_1], [H_2], \dots, [H_r]$ be the list of different conjugacy classes of subgroups in $ \Conj_G(X)$. Let $N_i:=N_G(H_i)$, and for any $K \leq G$, consider the $N_i$-conjugacy class $[K]_{N_i} :=\{ gKg^{-1} : g \in N_i \}$. Then,    
\[ \Rank( \End_G(X) : \Aut_G(X) ) = \sum_{i=1}^r \vert U(H_i) \vert  - \kappa_G(X),  \]
where $U(H_i) := \{ [G_x]_{N_i} : x \in X, H_i \leq G_x \}$ and $\kappa_G(X) := \vert \{ i : \vert \mathcal{O}_{[H_i]} \vert = 1 \} \vert$.
\end{theorem}

In order to prove Theorem \ref{main-intro}, we begin by discussing few basic results that hold for an arbitrary $G$-set and are probably known in the folklore of $G$-equivariant maps, but we could not find them published anywhere. For example, we observe that when the action of $G$ on $X$ is transitive, $\Aut_G(X) = \End_G(X)$ if and only if a point stabilizer in $X$ is not properly contained in any of its conjugates. Hence, the situation $\Aut_G(X) = \End_G(X)$ always occurs for transitive actions where point stabilizers are finite, or finite-index, or normal subgroups of $G$. Besides this, for a general $G$-set, we use the Imprimitive Wreath Product Embedding Theorem \cite[Theorem 5.5]{PS} to decompose $\Aut_G(X)$ as a direct product of wreath products:   
\[ \Aut_{G}(X) \cong \prod_{[H] \in \Conj_G(X) } ((N_G(H)/H) \wr \Sym(\mathcal{O}_{[H]}) ). \]

One of the motivations for our study comes from the \emph{shift action} of $G$ on $A^G$, where $A$ is a set, and $A^G$ is the set of all functions $x : G \to A$. This action, which has fundamental importance in symbolic dynamics and the theory of cellular automata (see \cite{Cecc}), is defined as follows: for every $g \in G$ and $x \in A^G$, 
\[ (g \cdot x)(h) := x(g^{-1}h), \quad \forall h \in G.  \]
When $A^G$ is considered as a topological space with the product topology of the discrete topology of $A$, the $G$-equivariant continuous transformations of $A^G$ turn out to be precisely the \emph{cellular automata} over $A^G$ (see the Curtis-Hedlund Theorem in \cite[Ch. 1]{Cecc}). Hence, the monoid $\CA(G;A)$, consisting of all cellular automata over $A^G$, is a submonoid of $\End_G(A^G)$. When $G$ and $A$ are both finite, it is clear that $\CA(G;A) = \End_G(A^G)$. This setting has been studied in \cite{cas3}, and the present paper generalizes and refines many of the results obtained there. In particular, Theorem \ref{main-intro} significantly generalizes Theorem 7 in \cite{cas3}, where the the relative rank of $\End_G(A^G)$ modulo $\Aut_G(A^G)$ was determined only when $G$ is a finite \emph{Dedekind group} (i.e. all subgroups of $G$ are normal).

The structure of this paper is as follows. In Section 2, we provide all the basic results on $G$-equivariant transformations that we need in order to prove Theorem \ref{main-intro}; we believe some of these basic results may be of interest in their own right. In Section 3 we prove Theorem \ref{main-intro} in three steps. First, we show that the set $W$, consisting of all transformations $[x \mapsto y]$, with $Gx \neq Gy$ and $G_x \leq G_y$, defined in (\ref{collapsing}), generates $\End_G(X)$ modulo $\Aut_G(X)$ (Lemma \ref{le-gen}). Second, we show that a particular subset $V$ of $W$ also generates $\End_G(X)$ modulo $\Aut_G(X)$ (Lemma \ref{legen2}). Finally, we complete the proof by showing that any set that generates $\End_G(X)$ modulo $\Aut_G(X)$ must contain at least $\vert V \vert$ transformations (Theorem \ref{main}). The cardinality of $V$ is given by the formula of Theorem \ref{main-intro}.


\section{Basic results}

For the rest of this paper, let $G$ be a group acting on a set $X$. We assume that the reader is familiar with the basic results on group actions (e.g., see \cite[Sec. 2]{PS}). Besides introducing notation, in this section we prove some required basic results on $G$-sets that we may not find published anywhere else. 

Denote by $\Sym(X)$ the symmetric group on $X$. When $X$ is finite and $\vert X \vert = n$, we may write $\Sym_n$ instead of $\Sym(X)$. Denote by $Gx$ and $G_{x}$ the $G$-orbit and stabilizer of $x\in X$, respectively: 
\[ Gx := \{ g \cdot x : g \in G \} \quad \text{ and } \quad G_x := \{ g \in G : g \cdot x = x \}. \]
Two basic results are that the set of $G$-orbits form a partition of $X$, and that stabilizars are subgroups of $G$. 

We shall repeatedly use the \emph{conjugation action} of $G$ on the set of its subgroups; this is defined, for any $g \in G$ and any subgroup $H \leq G$, by $g \cdot H := gHg^{-1}$. The $G$-orbits of this action, denoted by $[H]$, are called \emph{conjugacy classes}, while the stabilizers, denoted by $N_G(H)$, are called \emph{normalizers}; in other words, for any $H \leq G$, 
\[ [H] := \{ gHg^{-1} : g \in G \} \quad \text{ and } \quad N_G(H) := \{ g \in G : g H g^{-1} = H \}. \]

For a subset $Y \subseteq X$, define the \emph{setwise} and \emph{elementwise} stabilizers of $Y$ as the subgroups $G_Y = \{ g \in G : g \cdot Y = Y \}$ and $G_{(Y)} = \{ g \in G : g \cdot y = y, \forall y \in Y \}$, respectively. 

The main object of this paper is the monoid of $G$-equivariant transformations of $X$, which is denoted by
\[ \End_G(X) := \{ \tau : X \to X : \tau(g \cdot x) = g \cdot \tau(x), \forall g \in G, x \in X \}. \]
Let $\Aut_G(X)$ be the group of units of $\End_G(X)$, i.e.
\[ \Aut_G(X) := \{ \tau \in \End_G(X) : \exists \sigma \in \End_G(X), \ \sigma \tau = \tau \sigma = \id \}.  \]
When $G$ is a subgroup of $\Sym(X)$, i.e. $G$ is a \emph{permutation group}, note that $\Aut_G(X)$ is equal to the centralizer of $G$ in $\Sym(X)$.  

It is clear that $G$-equivariant transformations map $G$-orbits to $G$-orbits, as $\tau(Gx) = G \tau(x)$, for all $x \in X$, $\tau \in \End_G(X)$. Moreover, it is easy to see that if $\tau \in \End_G(X)$ is bijective, then $\tau^{-1}$ is also $G$-equivariant, so $\Aut_G(X)  = \End_G(X) \cap \Sym(X)$. 

Both $\End_G(X)$ and $\Aut_G(X)$ act naturally on $X$ by evaluation: for any $\tau \in \End_G(X)$ and $x \in X$, this action is defined by $\tau \cdot x := \tau(x)$. 

A subset $Y \subset X$ is \emph{$G$-invariant} if $g \cdot y \in Y$, for all $y \in Y$. It is easy to prove that $Y \subseteq X$ is $G$-invariant if and only if $Y$ is a union of $G$-orbits. We may restrict the action of $G$ to $Y$, and consider the monoid $\End_G(Y)$ and the group $\Aut_G(Y)$.

\begin{lemma}\label{le-invariant}
For any $G$-invariant subset $Y \subseteq X$, the following hold:
\begin{enumerate}
\item $\End_G(Y)$ is isomorphic to a submonoid of $\End_G(X)$, and $\Aut_G(Y)$ is isomorphic to a subgroup of $\Aut_G(X)$. 
\item If $Y$ is also $\Aut_G(X)$-invariant (i.e., $\tau(y) \in Y$, for all $y \in Y$, $\tau \in \Aut_G(X)$), then $\Aut_G(Y)$ is isomorphic to a normal subgroup of $\Aut_G(X)$. 
\item $\Aut_G(Y) \cong \Aut_G(X)_{Y}/ \Aut_G(X)_{(Y)}$. 

\end{enumerate}
\end{lemma}
\begin{proof}
The monoid $\End_G(Y)$ is embedded in $\End_G(X)$ via the injective homomorphism $\Phi : \End_G(Y) \to  \End_G(X)$ given by
\[ \Phi(\tau) (x) := \begin{cases}
\tau(x) & \text{ if } x \in Y \\
x & \text{ otherwise.}
\end{cases}\] 
It is easy to check that indeed $\Phi(\tau) (x) \in  \End_G(X)$ using the fact that $g \cdot x \in Y$ if and only if $x \in Y$. By restricting $\Phi$ to $\Aut_G(Y)$, we show that $\Aut_G(Y)$ is embedded in $\Aut_G(X)$. 

Now fix $\sigma \in \Aut_G(X)$ and $\tau \in \Aut_G(Y)$, and assume that $Y$ is also $\Aut_G(X)$-invariant. Since $\sigma(y) \in Y$, for all $y \in Y$, then $\sigma \vert_Y \in \Aut_G(Y)$. Therefore $\sigma \Phi(\tau) \sigma^{-1} = \Phi (\sigma \vert_Y \tau \sigma^{-1} \vert_Y) \in \Phi(\Aut_G(Y))$, which shows that $\Phi(\Aut_G(Y))$ is normal in $\Aut_G(X)$. 

Finally, consider the restriction homomorphism $\psi : \Aut_G(X)_{Y} \to \Aut_G(Y)$ given by $\phi(\tau) = \tau \vert_Y$, for all $\tau \in \Aut_G(X)_{Y}$. This is a surjective homomorphism with kernel $\Aut_G(X)_{(Y)}$, so part (3) follows by the First Isomorphism Theorem.   
\end{proof}

We shall introduce notation for three particular $G$-equivariant transformations of $X$. Let $x,y \in X$ with $x \neq y$. 
\begin{enumerate}
\item When $G_x \leq G_y$, define $[x \mapsto y]: X \to X$ by
\begin{equation}\label{collapsing}
 [x \mapsto y](z) := \left\{ \begin{array}{cl} g\cdot y & if\ z=g\cdot x \\ z & \text{otherwise.}  \end{array}  \right.
\end{equation}
\item When $G_x = G_y$ and there exists $k \in G$ such that $y = k \cdot x$, define $\tau_{x,k} : X \to X$ by 
\begin{equation}\label{tranorb}
\tau_{x,k}(z)= \left\{ \begin{array}{cl} g k \cdot x & if \ z=g\cdot x \\ z & \text{otherwise.}  \end{array}  \right. 
\end{equation}
\item When $G_x = G_y$ and $Gx \neq Gy$, define $[x , y] : X \to X$ by
\begin{equation} \label{trans}
 [x , y](z) := \left\{ \begin{array}{cl} g\cdot y & if\ z=g\cdot x \\ g\cdot x & if\ z=g\cdot y \\ z & \text{otherwise.}  \end{array}  \right. 
\end{equation}
\end{enumerate}

Observe that the $G$-equivariant transformations given in (\ref{tranorb}) and (\ref{trans}) are bijections, and $\tau_{x,k}$ is defined if and only if $G_{x} = G_{k \cdot x} = k G_x k^{-1}$, which is true if and only if $k \in N_G(G_x)$. 

For any $\tau \in \End_G(X)$, it is easy to check that $G_x \leq G_{\tau(x)}$ for every $x \in X$. Using this simple fact combined with the previous constructions of $G$-equivariant transformations, we obtain the following result. 

\begin{lemma}\label{lema1}
Let $x,y\in X$. 
\begin{enumerate}
\item There exists $\tau\in \Aut_G(X)$ such that $\tau(x)=y$ if and only if $G_{x}=G_{y}$.

\item If $Gx \neq Gy$, there exists $\tau\in \End_G(X) \setminus \Aut_G(X)$ such that $\tau(x)=y$ if and only if $G_{x} \leq G_{y}$.  
\end{enumerate}
\end{lemma}

Recall that the action of $G$ on $X$ is \emph{transitive} if there exists a unique $G$-orbit, so $X=Gx$ for any $x \in X$; in other words, if for every $x,y \in X$ there exists $g \in G$ such that $y=g \cdot x$. If the action of $G$ on $X$ is not transitive, Lemma \ref{lema1} shows that $\Aut_G(X)$ is properly contained in $\End_G(X)$. On the other hand, when the action is transitive, the next result shows that $\Aut_G(X) =  \End_G(X)$ if and only if, for any $x \in X$, the stabilizer $G_x$ is not properly contained in any of its conjugates (i.e., $G_x \subseteq g G_x g^{-1}$ implies $G_x = g G_x g^{-1}$). This last condition may seem unusual. In general, if $H$ is a finite, finite-index, or normal subgroup of $G$, then $H$ cannot be properly contained in any of its conjugates; however, there are examples of infinite groups $G$ and infinite-index subgroups $H$ such that $H$ is properly contained in $gHg^{-1}$ for some $g \in G$ (see \cite{K12}).

\begin{proposition}\label{th-transitive}
Suppose that the action of $G$ on $X$ is transitive and let $x \in X$. The subgroup $G_x$ is not properly contained in any of its conjugates if and only if
\[ \End_G(X) = \Aut_G(X). \]
\end{proposition}
\begin{proof}
Assume that $G_x$ is not properly contained in any of its conjugates. Fix $\tau \in \End_G(X)$. It follows by transitivity and $G$-equivariance that $\tau$ must be surjective. Suppose that $\tau(x) = \tau(y)$, for some $x,y \in G$. By transitivity, there exists $g \in G$ such that $y = g \cdot x$. Then $\tau(x) = \tau(g \cdot x) = g \cdot \tau(x)$, implies that $g \in G_{\tau(x)}$. By the $G$-equivarance of $\tau$, we have $G_x \leq G_{\tau(x)}$, and again by transitivity there exists $h \in G$ such that $G_{\tau(x)} = h G_x h^{-1}$. Hence, by hypothesis, $G_x = G_{\tau(x)}$, so $g \in G_x$. Hence, $y = g \cdot x = x$, which shows that $\tau$ is injective. 

Conversely, suppose that there exists $x \in X$ and $h \in G$ such that $G_x < h G_x h^{-1}$. Define $\tau : Gx \to Gx$ as follows:
\[ \tau( g \cdot x) := gh \cdot x, \quad \forall g \in G.  \]
It is clear that $\tau \in \End_G(X)$, but we shall show that $\tau$ is not injective. Consider $k \in (h G_x h^{-1}) - G_x$. Then $k \cdot x \neq x$. However, $k \in h G_x h^{-1} = G_{h\cdot x}$ implies that
\[  k \cdot (h\cdot x) = h \cdot x \ \Leftrightarrow \ \tau(k \cdot x) = \tau(x).   \]
Therefore, $\Aut_G(X)$ is properly contained in $\End_G(X)$. 
\end{proof}

\begin{example}
Let $G$ be an infinite group such that there exists a subgroup $K \leq G$ and $h \in G$ with $K < h K h^{-1}$. In this situation, consider the shift action of $G$ on $A^{G}$, where $A$ is a set with at least two elements. Without loss, assume that $\{0,1\} \subseteq A$. Let $x := \chi_K \in A^G$ be the \emph{indicator function} of $K$ defined by
\[ \chi_K (g) = \begin{cases}
1 & \text{ if } g \in K \\
0 & \text{ otherwise.}
\end{cases} \]
It is easy to check that $G_x = K$, and so, by Proposition \ref{th-transitive}, $\Aut_G( Gx)$ is properly contained in $\End_G(Gx)$.    
\end{example}

\subsection{Structure of $\Aut_G(X)$}

The next result is well-known (for example, see \cite[Prop. 1.8]{Homo}).

\begin{lemma}\label{lemmaAut-orbit}
If the action of $G$ on $X$ is transitive, then
\[  \Aut_G(X) \cong N_G(G_x)/G_x, \text{ for any } x \in X. \]
\end{lemma}
\begin{proof}
Fix $x \in X$. The map $k \mapsto \tau_{x,k}$ from $N_G(G_x)$ to $\Aut_G(X)$, with $\tau_{x,k}$ defined as in Lemma \ref{lema1}, is a surjective homomorphism with kernel $G_x$. 
\end{proof}

Our next goal is to describe the structure of $\Aut_G(X)$ for intransitive actions. 

First of all, we introduce some notation. Define 
\[ \Stabs_G(X) := \{ G_x : x \in X \} \ \text{ and }  \  \Conj_G(X) := \{[G_x] : x \in X \}.\]
For each $[H] \in \Conj_G(X)$, define the set
\[ \mathcal{B}_{[H]}:=\{x\in X :\ [G_{x}]=[H]\}. \]
Since $G_{g \cdot x} = gG_xg^{-1}$, for every $g \in G$, $x \in X$, the set $\mathcal{B}_{[H]}$ is $G$-invariant. Moreover, by Lemma \ref{lema1}, $\mathcal{B}_{[H]}$ is also $\Aut_G(X)$-invariant, so, by Lemma \ref{le-invariant}, $\Aut_G(\mathcal{B}_{[H]})$ is isomorphic to a normal subgroup of $\Aut_G(X)$. The number of $\Aut_G(X)$-orbits inside $\mathcal{B}_{[H]}$ is $\vert [H]\vert = [G : N_{G}(H)]$.

By the Orbit-Stabilizer Theorem, the cardinality of each $G$-orbit inside $\mathcal{B}_{[H]}$ is the index $[G:H]$. Let $\mathcal{O}_{[H]}$ be the set of $G$-orbits contained in $\mathcal{B}_{[H]}$:
\[ \mathcal{O}_{[H]} := \{Gx \subseteq X :\ [G_{x}]=[H], x \in X \}, \text{ and let } \alpha_{[H]} := \vert \mathcal{O}_{[H]}  \vert. \]

\begin{example}
For the shift action of $G$ on $A^G$, the values of $\alpha_{[H]}$ may be calculated using the lattice of finite-index subgroups $L(G)$ of $G$ and its corresponding M\"{o}bius function $\mu : L(G) \times L(G) \to \mathbb{Z}$; see \cite{CS21} for details. 
\end{example}

\begin{example}
Consider the shift action of $\mathbb{Z}_4$ on $\{0,1 \}^{\mathbb{Z}_4}$, where we represent a function $x \in \{0,1 \}^{\mathbb{Z}_4}$ by the $4$-tuple $(x(0), x(1), x(2), x(3))$. In this case, $\Stabs_G(X) = \{ H_1 := \langle 1 \rangle, H_2 :=  \langle 2 \rangle , H_3 := \langle 0 \rangle \}$. Figure \ref{fig} shows all the elements of $\{0,1 \}^{\mathbb{Z}_4}$ partitioned by $\mathbb{Z}_4$-orbits (depicted by ellipses) and by the sets $\mathcal{B}_{[H_i]}$ (depicted by rectangles). In this case, we directly see that $\alpha_{[H_1]} =2$, $\alpha_{H_2} = 1$ and $\alpha_{[H_3]} = 3$. 
\begin{figure}[h]
 \centering
\includegraphics[scale=0.5]{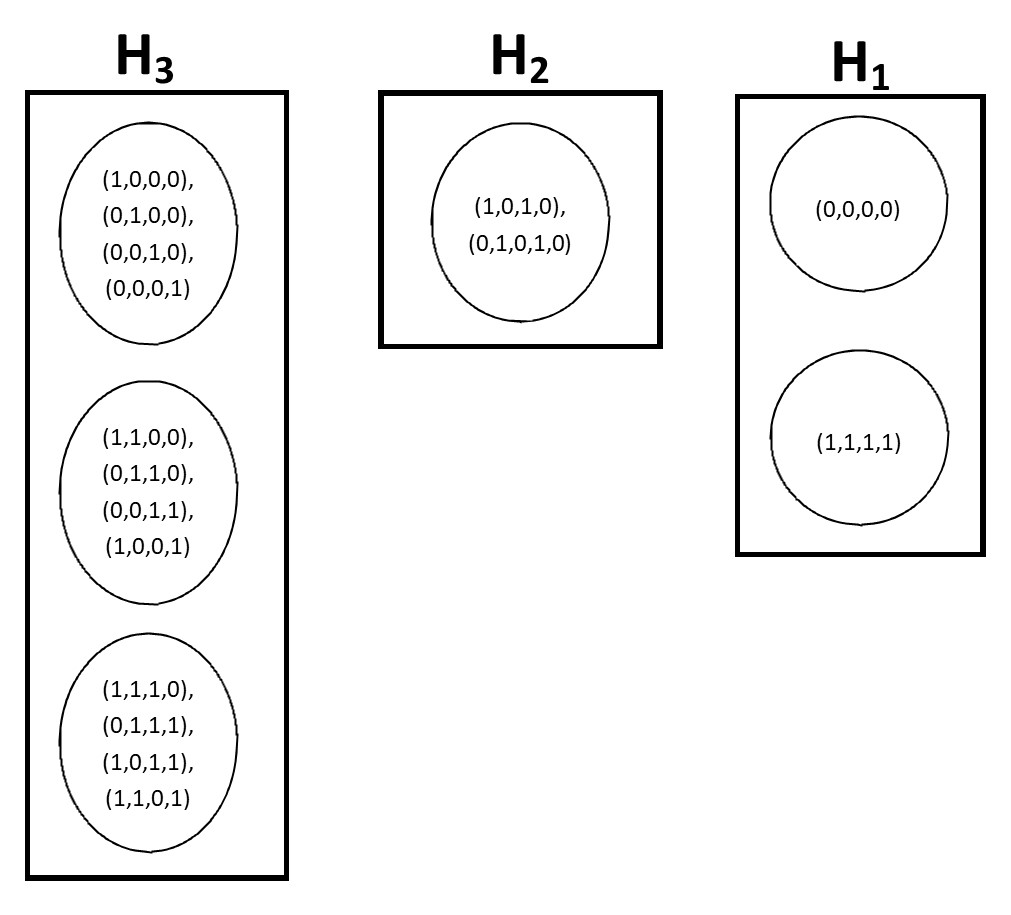}
\caption{Partitions for the shift action of $\mathbb{Z}_4$ on $\{0,1 \}^{\mathbb{Z}_4}$}
\label{fig}
\end{figure}
\end{example}


We briefly recall the definition of wreath product of groups. For a group $K$, a set $\Delta$ and a group $H$ acting on $\Delta$, the \emph{(unrestricted) wreath product} of $K$ by $H$, denoted by $K \wr_{\Delta} H$ is the semidirect product $K^{\Delta} \rtimes_\phi H$, where $K^\Delta$ is the direct product of $G$ indexed by $\Delta$ and homomorphism $\phi : H \to \Aut(K^\Delta)$ given by $\phi(h)(f)(\delta) = f(h \cdot \delta)$, for all $h \in H, f \in K^\Delta, \delta \in \Delta$ (see \cite[p. 105]{PS} for more details). In our results, we shall always consider the case when $H=\Sym(\Delta)$ and the action of $H$ on $\Delta$ is the natural action by evaluation, so, in order to simplify notation, we write $K \wr \Sym(\Delta)$ instead of $K \wr_{\Delta} \Sym(\Delta)$.   

For any set $A$, let $\Tran(A)$ be the \emph{full transformation monoid} of $A$, i.e. the set of all self-maps of $A$ equipped with composition. When $A$ is finite, it is well-known that $\Tran(A)$ is generated by a generating set of $\Sym(A)$ together with a map with image size $\vert A \vert - 1$ (e.g., see \cite[Ch. 3]{Gany}). Analogously to the case of groups, we may define a wreath product of monoids (see \cite{Ara}).

The following result uses the Imprimitive Wreath Product Embedding Theorem, which, as stated in \cite[Theorem 5.5]{PS}, involves a partition with finitely many blocks; however, this assumption may be dropped and the same proof of the theorem works by using the axiom of choice.

\begin{proposition} \label{le-box}
Let $G$ be a group acting on a set $X$, 
\begin{enumerate}
\item For any $H \in \Stabs_G(X)$,
\[  \Aut_G(\mathcal{B}_{[H]}) \cong \left( N_G(H)/H \right) \wr \Sym(\mathcal{O}_{[H]}).   \]

\item 
\[  \Aut_G(X) \cong \prod_{[H] \in \Conj_G(X) } ((N_G(H)/H) \wr \Sym(\mathcal{O}_{[H]}) ).  \]

\item Suppose that $H \in \Stabs_G(X)$ is not properly contained in any of its conjugates. Then, 
\[  \End_G(\mathcal{B}_{[H]}) \cong \left( N_G(H)/H \right) \wr \Tran(\mathcal{O}_{[H]}).   \]
\end{enumerate}
\end{proposition}
\begin{proof}
\begin{enumerate}
\item The set $\mathcal{O}_{[H]}$ of $G$-orbits inside $\mathcal{B}_{[H]}$ forms a $G$-invariant partition of $\mathcal{B}_{[H]}$, with $\vert \mathcal{O}_{[H]} \vert = \alpha_{[H]}$. Using the axiom of choice, it is easy to see that the induced action of $\Aut_G(\mathcal{B}_{[H]})$ on $\mathcal{O}_{[H]}$ is isomorphic to the whole symmetric group $\Sym(\mathcal{O}_{[H]})$ (pick representatives of $G$-orbits in order to define corresponding $G$-equivariant transformations). By \cite[Theorem 5.5]{PS}) and Lemma \ref{le-invariant}, $\Aut_G(\mathcal{B}_{[H]})$ is permutationally isomorphic to a subgroup $R$ of $\Aut_G(Gx) \wr \Sym(\mathcal{O}_{[H]})$, with $Gx \in \mathcal{O}_{[H]}$. The kernel of the projection of $R$ to $\Sym(\mathcal{O}_{[H]})$ is isomorphic to the direct product $\Aut(Gx)^{\mathcal{O}_{[H]}}$. Therefore,
\[ R= \Aut_G(Gx) \wr \Sym(\mathcal{O}_{[H]}),  \] 
and the result follows by Lemma \ref{lemmaAut-orbit}.

\item For each $\tau \in \Aut_G(X)$, we define a function 
\[ F : \Aut_G(X) \to \prod_{[H] \in \Conj_G(X) }\Aut_G(\mathcal{B}_{[H]}) \]
by $F(\tau)_{[H]} := \tau \vert_{\mathcal{B}_{[H]}}$ for every $\tau \in \Aut_G(X)$. The function $F$ is injective because $\{ \mathcal{B}_{[H]} : [H] \in \Conj_G(X)  \}$ is a  partition of $X$. To show that $F$ is surjective, observe that for any $( \tau_{[H]} : [H] \in \Conj_G(X))$ in the product, we may define $\tau \in \Aut_G(X)$ by $\tau(x) = \tau_{[H]}(x)$ if and only if $x \in \mathcal{B}_{[H]}$. It follows that $\tau$ is indeed $G$-equivariant as, for every $[H] \in \Conj_G(X)$, $\tau_{[H]}$ is $G$-equivariant and $\mathcal{B}_{[H]}$ is $G$-invariant. Finally, $F$ is a homomorphism since $\mathcal{B}_{[H]}$ is $\Aut_G(X)$-invariant so $(\tau \circ \sigma)\vert_{\mathcal{B}_{[H]}} = \tau\vert_{\mathcal{B}_{[H]}} \circ \sigma \vert_{\mathcal{B}_{[H]}}$, for every $\tau, \sigma \in \Aut_G(X)$. The result follows by Lemma \ref{le-box}.

\item With the help of Proposition \ref{th-transitive}, this follows by a very similar argument as in the proof of the first point. 
\end{enumerate}
\end{proof}


\section{The relative rank of $\End_G(X)$ modulo $\Aut_G(X)$}

For any monoid $M$ and $U \subseteq M$, we say that a subset $W$ of $M$ \emph{generates $M$ modulo} $U$ if $\langle W \cup U \rangle = M$. 

For this section, assume that $G$ is a finite group acting on a finite set $X$. We shall assume that $[H_1], [H_2], \dots, [H_r]$ is the list of different conjugacy classes of subgroups in $\Conj_G(X)$. Define $\mathcal{B}_i := \mathcal{B}_{[H_i]}$ and $\alpha_i := \alpha_{[H_i]}$. For $n \in \mathbb{N}$, denote $[n] := \{1 , 2 , \dots, n \}$.

\begin{lemma}\label{le-gen}
The set
\[W := \left\{ [ x \mapsto y] : x,y \in X, Gx \neq Gy, G_x \leq G_y \right\} \]
generates $\End_G(X)$ modulo $\Aut_G(X)$
\end{lemma}
\begin{proof}
Suppose that the list of conjugacy classes is ordered such that
\[ \vert H_1 \vert \geq \vert H_2 \vert \geq \dots \geq \vert H_r \vert.  \]
Fix $\tau \in \End_G(X)$ and for each $i \in [r]$, define $\tau_i \in \End_G(X)$ by
\[ \tau_i(x)  :=\begin{cases}
\tau(x) & \text{ if } x \in \mathcal{B}_i \\
x & \text{ otherwise. }
\end{cases}\]
Each $\tau_i$ is indeed $G$-equivariant as $\mathcal{B}_i$ is $G$-invariant. By Lemma \ref{lema1}, $G_x \leq G_{\tau(x)}$ for all $x \in X$, so $\tau(\mathcal{B}_i) \subseteq \bigcup_{j \leq i} \mathcal{B}_j$. This implies that we have the following factorization of $\tau$
\begin{equation}
 \tau = \tau_r \circ \tau_{r-1} \circ \dots \circ \tau_2 \circ \tau_1,
\end{equation}
which holds since $\tau_i (x)$ is fixed by $\tau_j$, for all $x \in \mathcal{B}_i$, $j > i$. Now fix $i \in [ r]$ and define
\begin{align*}
\mathcal{B}_i^0 & := \{ x \in \mathcal{B}_i :  \tau(x) \in \mathcal{B}_i \}, \\
\mathcal{B}_i^1 & := \{ x \in \mathcal{B}_i : \tau(x) \in \mathcal{B}_j \text{ with } j<i \}.
\end{align*}
For $\epsilon \in \{0,1\}$, it is easy to see that $\mathcal{B}_i^\epsilon$ is a $G$-invariant subset of $X$, so we may define $\tau_i^{\epsilon} \in \End_G(X)$ by 
\[ \tau_i^{\epsilon}(x)  :=\begin{cases}
\tau_i(x) & \text{ if } x \in \mathcal{B}_i^{\epsilon} \\
x & \text{ otherwise. }
\end{cases}\] 
It follows that $\tau_i = \tau_i^{0} \circ \tau_i^{1}$. We finish the proof by showing that $\tau_{i}^\epsilon \in \langle W \cup \Aut_G(X) \rangle$, for $\epsilon \in \{0,1\}$. 
\begin{itemize}
\item \emph{Case $\epsilon = 0 $}. We show that $\tau_i^0 \in \langle W \cup \Aut_G(X) \rangle$. In this case, $\tau_i^0$ is contained in a submonoid of $\End_G(X)$ isomorphic to $\End_G(\mathcal{B}_i)$. If $\alpha_i = 1$, then $\End_G(\mathcal{B}_i) = \End_G(Gx) = \Aut_G(Gx)$, for $x \in \mathcal{B}_i$, so assume that $\alpha_i > 2$. By Proposition \ref{le-box}, $\End_G(\mathcal{B}_i) \cong \Aut_G(Gx) \wr \Tran(\mathcal{O}_{[H_i]})$, for any $x \in \mathcal{B}_i$. The monoid $\Tran( \mathcal{O}_{[H_i]})$ is generated by $\Sym(\mathcal{O}_{[H_i]})$ together with any mapping with image size $\alpha_i - 1$ (see \cite[Prop. 1.2]{HRH98}). For any $x, y \in \mathcal{B}_i$ with $Gx \neq Gy$, the map $[x \mapsto y]$ induces on $\mathcal{O}_{H_i}$ a mapping with image size $\alpha_i-1$. It follows that $\End_G(\mathcal{B}_i)$ is generated by $\Aut_G(\mathcal{B}_i) \cup \{[ x \mapsto y ] \vert_{\mathcal{B}_i} \}$, and the result follows. 

\item \emph{Case $\epsilon = 1 $}. We show that $\tau_i^1 \in \langle W \cup \Aut_G(X) \rangle$. Suppose that $\mathcal{B}_i^1$ is a union of the following $G$-orbits
\[  \mathcal{B}_i^1 = Gx_1 \cup Gx_2 \cup \dots G x_s,  \]
for some $x_1, \dots, x_s \in X$. Write $y_k := \tau^1_i(x_k)$, for $k \in \{1, \dots, s \}$. As $y_k\not\in \mathcal{B}_i$, then $Gx_k \neq Gy_k$. Therefore, we have the factorizacion
\[ \tau_i^1 = [x_1 \mapsto y_1][x_2 \mapsto y_2] \dots [x_s \mapsto y_s] \in \langle W \cup \Aut_G(X) \rangle.  \]
\end{itemize}
\end{proof}

For any $N, H \leq G$, define the $N$-conjugacy class of $H$ by $[H]_N := \{ gHg^{-1} : g \in N\}$. 

\begin{lemma}\label{legen2}
Let $G$ be a finite group acting on a finite set $X$. For each $i \in [r]$, $N_i := N_G(H_i)$, and let $[K_{i,1}]_{N_i}, \dots, [K_{i,r_i}]_{N_i}$ be the list of $N_i$-conjugacy classes such that $K_{i,j} \in \Stabs_G(X)$ and $H_i < K_{i,j}$. We fix some elements of $X$ as follows:
\begin{itemize}
\item For each $i \in [r]$, fix $x_i \in \mathcal{B}_i$ such that $G_{x_i} = H_i$.
\item For each $i \in [r]$ and $j \in [r_i]$, fix $y_{i,j} \in X$ such that $G_{y_{i,j}} = K_{i,j}$. 
\item For each $i \in [r]$ such that $\alpha_i \geq 2$, fix $x_i^\prime \in \mathcal{B}_i$ such that $G_{x_i^\prime} = H_i$ and $Gx_i \neq Gx_i^\prime$. 
\end{itemize}
Then the set 
\[ V :=  \left\{ [x_i \mapsto y_{i,j} ] : i \in [r], j \in [r_i] \right\} \cup \{ [x_i \mapsto x_i^\prime] : i \in [r], \alpha_i \geq 2 \}. \]
generates $\End_G(X)$ modulo $\Aut_G(X)$.   
\end{lemma}
\begin{proof}
By Lemma \ref{le-gen}, it is enough to show that any $[x \mapsto y]$, with $x,y \in X$, $Gx \neq Gy$, $G_x \leq G_y$, may be expressed as a product of elements of $V$ and $\Aut_G(X)$. Without loss, assume that $[G_x] = [H_i]$. Then, there exists $g \in G$ such that $G_{g \cdot x}=gG_x g^{-1} = H_i$; since $[x \mapsto y ] = [ (g\cdot x) \mapsto (g \cdot y)]$, we may assume that $G_x = H_i$. There are two cases to consider.
\begin{itemize}
\item \emph{Case $G_x = G_y$}. Note that $G_x = H_i = G_{x_i} = G_{x_i^\prime} = G_y$. We have four subcases:
\begin{itemize}
\item \emph{Case $Gx \neq Gx_i$ and $Gy \neq Gx_i^\prime$}. Here we have the factorization
\[ [x \mapsto y] = [ y,x_i^\prime][x, x_i]  [x_i \mapsto x_i^\prime] [x, x_i] [ y,x_i^\prime] \in \langle V \cup \Aut_G(X) \rangle.  \]

\item \emph{Case $Gx = Gx_i$ and $Gy \neq Gx_i^\prime$}. There exists $h \in G$ with $x_i = h \cdot x$. As $G_x = G_{x_i}$, we may define $\tau_{x,h} \in \Aut_G(X)$ as in the proof of Lemma \ref{lema1}. Then
\[ [x \mapsto y] = [x_i \mapsto y ] \tau_{x, h}, \]
which holds since for every $g \in G$,
\[ [x_i \mapsto y ] \tau_{x,h}( g \cdot x) = [x_i \mapsto y ] ( gh \cdot x) = [x_i \mapsto y ] ( g \cdot x_i) = g \cdot y.  \]
As $Gy \neq Gx_i^\prime$, then 
\[ [x \mapsto y] =[ y,x_i^\prime] [x_i \mapsto x_i^\prime ] [ y,x_i^\prime] \tau_{x, h} \in \langle V \cup \Aut_G(X) \rangle.  \]

\item \emph{Case $Gx \neq Gx_i$ and $Gy = Gx_i^\prime$}. There exists $k \in G$ with $y = k \cdot x_i^\prime$. Observe that
\[  [x \mapsto y] = \tau_{x_i^\prime,k} [x \mapsto x_i^\prime] \tau_{x_i^\prime, k^{-1}}. \]
Therefore,
\[ [x \mapsto y]  = \tau_{x_i^\prime,k} [x, x_i ] [x_i \mapsto x_i^\prime] [x,x_i] \tau_{x_i^\prime, k^{-1}} \in \langle V \cup \Aut_G(X) \rangle. \]

\item \emph{Case $Gx = Gx_i$ and $Gy = Gx_i^\prime$}. With the notation of the above cases,
\[ [x \mapsto y ] =  \tau_{x_i^\prime,k} [x_i \mapsto x_i^\prime] \tau_{x, h}\tau_{x_i^\prime,k^{-1}} \in \langle V \cup \Aut_G(X) \rangle.  \]
\end{itemize}

\item \emph{Case $G_x <  G_y$}. There exists $K_{i,j}= G_{y_{i,j}}$ such that $[K_{i,j}]_{N_i} = [G_y]_{N_i}$, with $N_i:= N_G(H_i)$. Thus, there exists $n \in N_i$ such that 
\[ K_{i,j} = n G_y n^{-1} = G_{z}, \]
where $z:= n \cdot y \in X$. Now, $G_x = G_{n \cdot x}$ because $n \in N_i$, and we may define $\tau_{x, n} \in \Aut_G(X)$. Hence,
\[[x \mapsto y] = [ (n^{-1} \cdot x) \mapsto y] \tau_{x, n^{-1}} =  [x \mapsto (n \cdot y) ]\tau_{x, n^{-1}} = [x \mapsto z]  \tau_{x, n^{-1}} . \] 
This shows that $[x \mapsto y] \in \langle V \cup \Aut_G(X) \rangle$ if and only if $[x \mapsto z] \in \langle V \cup \Aut_G(X) \rangle$, with $G_z = K_{i,j}  =  G_{y_{i,j}}$. Again there are four subcases, which are done analogously as in the previous case. For example, for the subcase $Gx \neq Gx_i$ and $Gz \neq Gy_{i,j}$, we obtain the factorization
\[ [x \mapsto z ] = [x, x_{i}]  [z, y_{i,j}] [x_i \mapsto y_{i,j} ] [z, y_{i,j}] [x, x_i] \in  \langle V \cup \Aut_G(X) \rangle . \]
\end{itemize}
\end{proof}

With the notation of the previous lemma, define 
\[ U(H_i) :=  \{ [K]_{N_i} : K \in \Stabs_G(X), H_i \leq K \} \text{ and } \kappa_G(X) := \vert \{ i : \alpha_{i} = 1 \} \vert. \]

\begin{corollary}\label{cor-low}
With the previous notation,
\[  \Rank( \End_G(X) : \Aut_G(X) ) \leq \vert V \vert = \sum_{i=1}^r \vert U(H_i) \vert  - \kappa_G(X).  \]
\end{corollary}

\begin{example}
When $G$ is a finite group acting by the shift action on $A^G$, with $A$ a finite set with at least two elements, it is not hard to show (see \cite[Lemma 5]{cas3}) that 
\[ \kappa_G(A^G) = \begin{cases} \vert \{ i : \vert G/H_i \vert = 2 \} \vert & \text{ if } \vert A \vert = 2,  \\
0 & \text{otherwise.}
\end{cases}\] 
\end{example}

Before being able to determine the exact relative rank of $\End_G(X)$ modulo $\Aut_G(X)$, we need to introduce few definitions. For $\tau \in \End_G(X)$, define 
\[ \ker(\tau)  = \{ (a,b) \in X \times X : \tau(a) = \tau(b) \}.   \]
This is an equivalence relation on $X$, and it is linked with the $\mathcal{L}$ Green's relation in a semigroup of transformations (see \cite[Theorem 4.5.1]{Gany}).

\begin{remark}\label{kernel}
We state the following elementary properties of kernels of transformations:
\begin{enumerate}
\item $\ker(\tau) = \{ (a,a) : a \in X \}$ if and only if $\tau$ is a bijection. 
\item $\tau, \sigma \in \End_G(X)$, we have $\ker(\tau) \subseteq \ker(\sigma \tau)$. 
\item If $\tau \in \Aut_G(X)$, then $\vert \ker(\sigma) \vert = \vert \ker(\sigma\tau) \vert$, since $(a,b) \in \ker(\sigma \tau)$ if and only if $(\tau(a), \tau(b)) \in \ker(\sigma)$. 
\end{enumerate}
\end{remark}

\begin{definition}\label{def-elemcop}
Let $K \in \Stabs_G(X)$ be such that $H_i \leq K \leq G$. We say that $\tau \in \End_G(X)$ is an \emph{elementary collapsing of type $(i, [K]_{N_i})$} if there exist $x,y\in X$, with $Gx \neq Gy$, such that:
\begin{enumerate}
 \item $G_x = H_i$ and $[G_{\tau(x)}]_{N_i} = [G_y]_{N_{i}}= [K]_{N_i}$.
 \item $\ker(\tau) = \{ (g \cdot x, g \cdot y), (g \cdot y, g \cdot x): g \in G \} \cup \{ (a,a) : a \in X \}$.
 \end{enumerate}
\end{definition}

For example, when $x,y \in X$ satisfy that $G_x = H_i$ and $[G_y]_{N_i} =[K]_{N_i} $, then $[x \mapsto y]$ is an elementary collapsing of type $(i, [K]_{N_i})$.

\begin{lemma}\label{welldef}
\begin{enumerate}
\item Suppose $\tau \in \End_G(X)$ is an elementary collapsing of type $(i,[K]_{N_i})$ and type $(i^\prime, [K^\prime]_{N_i})$. Then $i=i^\prime$ and $[K]_{N_i}=[K^\prime]_{N_i}$. 
\item The number of possible types of elementary collapsings is precisely
\[ \sum_{i=1}^r U(H_i)  - \kappa_G(X). \]
\end{enumerate}
\end{lemma}
\begin{proof} 
For the first part, there exist $x, x^\prime, y, y^\prime \in X$ with $Gx \neq Gy$ and $Gx^\prime \neq Gy^\prime$ satisfying Definition \ref{def-elemcop}. Then
\[  \{ (g \cdot x, g \cdot y), (g \cdot y, g \cdot x) : g \in G \} =  \{ (g \cdot x^\prime, g \cdot y^\prime), (g \cdot y^\prime, g \cdot x^\prime) : g \in G \}. \]
This shows that $Gx = Gx^\prime$ and $Gy = Gy^\prime$. The first equality implies that $[H_i]=[G_x] = [G_{x^\prime}]=[H_{i^\prime}]$, so $i=i^\prime$. Moreover, there exists $g \in G$ such that $x = g \cdot x^{\prime}$, so 
\[ H_i=G_{x} = G_{g \cdot x^\prime} = gG_{x^\prime} g^{-1} = g H_i g^{-1} \ \Rightarrow \ g \in N_i = N_G(H_i) .\]  
Now,
\[  G_{\tau(x)} = G_{\tau(g \cdot x^\prime)} = G_{g \cdot \tau(x^\prime)} = g G_{\tau(x^\prime)} g^{-1}, \]
which implies that $[G_{\tau(x)}]_{N_i} = [G_{\tau(x^{\prime})}]_{N_i} $, and so $[K]_{N_i}=[K^\prime]_{N_i}$

The second part of the lemma follows as an elementary collapsing of type $(i, [H_i]_{N_i})$ exists if and only if $\alpha_{[H_i]} \neq 1$.
\end{proof}

\begin{theorem}\label{main}
Let $G$ be a finite group acting on a finite set $X$. Let $[H_1], [H_2], \dots, [H_r]$ be the list of different conjugacy classes of subgroups in $\Conj_G(X)$. Then, 
\[ \Rank( \End_G(X) : \Aut_G(X) ) = \sum_{i=1}^r U(H_i)  - \kappa_G(X).  \]  
\end{theorem}
\begin{proof}
Let $W \subseteq \End_G(X)$ such that $\langle \Aut_G(X) \cup W \rangle = \End_G(X)$. We shall show that $W$ must contain at least one elementary collapsing of each possible type. Suppose that $K \in \Stabs_G(X)$ satisfies that $H_i \leq K \leq G$. Let $x,y \in X$ be such that $G_x = H_i$ and $[G_y]_{N_i} = [K]_{N_i}$. Since $[x \mapsto y] \in \End_G(X)$, there exist $\tau_1, \tau_2, \dots, \tau_s \in \Aut_G(X) \cup W$ such that
\[ [x \mapsto y] = \tau_s \tau_{s-1} \dots \tau_2 \tau_1. \]

Let $k$ be the smallest index such that $\tau_k$ is non-invertible. We must have $\tau_k \in W$. 

\begin{claim}\label{claim-key}
For all $m \in [r]- \{ i \}$, we have:
\begin{enumerate}
\item $\tau_k(\mathcal{B}_{m}) = \mathcal{B}_m$, and $G_{z} = G_{\tau_k(z)}$ for every $z \in \mathcal{B}_m$.
\item If $y \in \mathcal{B}_i$, then $\tau_k(\mathcal{B}_{i}) \subseteq \mathcal{B}_i$, and $G_z = G_{\tau_k(z)}$ for every $z \in \mathcal{B}_i$. 
\end{enumerate}
\end{claim}
\begin{proof}
Suppose there exists $w \in \mathcal{B}_m$ such that $\tau_k(w) \not\in \mathcal{B}_m$. Let $z := \tau_1^{-1} \dots \tau_{k-1}^{-1}(w)$. By Lemma \ref{lema1}, $z \in \mathcal{B}_m$, and so $[x \mapsto y] (z) = z$. This implies that $\tau_s \dots \tau_{k+1}$ must map $\tau_k(w) \not \in \mathcal{B}_m$ to an element of $\mathcal{B}_m$, which is impossible by Lemma \ref{lema1}. This proves that $\tau_k(\mathcal{B}_{m}) \subseteq \mathcal{B}_m$. If $\tau_k(\mathcal{B}_{m})$ is properly contained in $\mathcal{B}_m$, then there must exist $w_1, w_2 \in \mathcal{B}_m$ such that $\tau_k(w_1) = \tau_k(w_2)$. However, this implies that $z_1 := \tau_1^{-1} \dots \tau_{k-1}^{-1}(w_1) \in \mathcal{B}_m$ and $z_2 := \tau_1^{-1} \dots \tau_{k-1}^{-1}(w_2) \in \mathcal{B}_m$, with $z_1 \neq z_2$, satisfy that $[x \mapsto y](z_1) = [x \mapsto y](z_2)$, which contradicts the definition of $[x \mapsto y]$. Thus $\tau_k(\mathcal{B}_{m}) = \mathcal{B}_m$, and $G_{z} = G_{\tau_k(z)}$ for every $z \in \mathcal{B}_m$, follows by Lemma \ref{lema1} and the finiteness of $G$. 

The second part of the claim is done analogously. 
\end{proof}

We shall show that $\tau_k$ is an elementary collapsing of type $(i, [K]_{N_i})$ by verifying the two properties of Definition \ref{def-elemcop}.
\begin{enumerate}
\item As $\tau_k$ is non-invertible, there exist $a,b \in X$, $a \neq b$, such that $\tau_k(a) = \tau_k(b)$. Let $a^\prime := \tau_1^{-1} \dots \tau_{k-1}^{-1}(a)$ and $b^\prime := \tau_1^{-1} \dots \tau_{k-1}^{-1}(b)$. Then
\[ \tau_s \tau_{s-1} \dots \tau_2 \tau_1(a^\prime) = \tau_s \dots \tau_k(a) = \tau_s \dots \tau_k(b) = \tau_s \tau_{s-1} \dots \tau_2 \tau_1(b^\prime), \]
implies that $(a^\prime,b^\prime) \in \ker([x \mapsto y])$, so, without loss, there exists $g \in G$ such that 
\begin{align*}
 x& = g \cdot a^\prime = g \cdot \tau_1^{-1} \dots \tau_{k-1}^{-1}(a) =\tau_1^{-1} \dots \tau_{k-1}^{-1}(g \cdot a)  \\
 y & = g \cdot b^\prime = g \cdot \tau_1^{-1} \dots \tau_{k-1}^{-1}(b) = \tau_1^{-1} \dots \tau_{k-1}^{-1}(g \cdot b). 
\end{align*} 
By Lemma \ref{lema1}, $H_i = G_x = G_{g \cdot a}$ and $G_y = G_{g \cdot b}$. By $G$-equivariance, $\tau_k (g \cdot a) = \tau_k (g \cdot b)$. By Claim \ref{claim-key}, $G_{\tau_k(g \cdot b)} = G_{g \cdot b} = G_y$, so $G_{\tau_k(g \cdot a)} = G_y$. Hence, 
\[ [G_{\tau_k(g \cdot a)}]_{N_i} = [G_{g \cdot b}]_{N_i} = [G_y]_{N_i} = [K ]_{N_i},\]
 and $\tau_k$ satisfies part (1.) of Definition \ref{def-elemcop} with $g \cdot a$ and $g \cdot b$.

\item  By $G$-equivariance, $\tau_k(h \cdot a) = \tau_k(h \cdot b)$ for all $h \in G$.
By Remark \ref{kernel} (2.), 
\[  \ker(\tau_k \tau_{k-1} \dots \tau_2 \tau_1 ) \subseteq \dots  \ker(\tau_s \dots \tau_1) = \ker([x \mapsto y]). \]  
Moreover, as $\tau_{k-1} \dots \tau_1 \in \Aut_G(X)$, we have by Remark \ref{kernel} (3.), that
\[ \vert \ker(\tau_k) \vert = \vert \ker(\tau_k \tau_{k-1} \dots \tau_2 \tau_1) \vert \leq \vert \ker([x \mapsto y]) \vert. \]
Therefore, $\ker(\tau_k) = \{ (h \cdot a, h \cdot b), (h \cdot b, h \cdot a) : g \in G \} \cup \{ (c,c) : c \in X\}$. 
\end{enumerate}

This shows that $W$ must contain an elementary collapsing of each possible type, and the result follows by Lemma \ref{welldef} and Corollary \ref{cor-low}.
\end{proof}


\section*{Acknowledgments}

The first and second author were supported by a CONACYT Basic Science Grant (No. A1-S-8013) and a PhD CONACYT National Scholarship, respectively. Both authors sincerely thank Csaba Schneider for his kind and enriching replies to our questions on the Imprimitive Wreath Product Embedding Theorem. We also thank the anonymous referee of this paper for his insightful comments.  


\end{document}